\documentclass[reqno]{amsart}
\usepackage{amsthm,amssymb,amsmath}
\usepackage{mathrsfs, euscript}
\usepackage{hyperref}
\usepackage{tikz}
\usetikzlibrary[matrix, arrows, arrows.meta, calc, decorations.pathmorphing, backgrounds, positioning, fit, petri, cd, 3d, shadings, intersections, patterns.meta, shadows, knots, decorations.pathreplacing]
\usetikzlibrary{arrows.meta}
\newcommand{\C}{\mathbb{C}}  % the complex numbers
\newcommand{\R}{\mathbb{R}}  % the real numbers
\newcommand{\Z}{\mathbb{Z}}  % the integers 

\newcommand{\eu}[1]{\EuScript{#1}}

\newcommand{\msf}[1]{\mathsf{#1}}

\newcommand{\op}[1]{\operatorname{\msf{#1}}}
\newcommand{\euH}{\EuScript{H}}
\newcommand{\euK}{\EuScript{K}}

\newcommand{\trr}{\triangleright}

\swapnumbers
\newtheorem{theorem}{Theorem}
\newtheorem{lemma}[theorem]{Lemma}

\theoremstyle{definition}
\newtheorem{definition}[theorem]{Definition}
\newtheorem{example}[theorem]{Example}
\newtheorem{topic}[theorem]{}
\theoremstyle{remark}
\newtheorem{remark}[theorem]{Remark}
\def \fnf(#1, #2){ (.8 + (#2) * (#1 + 2) + (2* (#1)) * ( (#1)*(#1) + 1)^(-1))/2.8   }

\def \fns(#1){  1.5 - (sin(180 * (#1))/2.7) }

\begin{document}
\title{Fundamental racks of braid spaces of complex reflection groups}
\author{Tathagata Basak}
\thanks{Supported by Simons Foundation Collaboration Grant 637005.}
\address{Department of Mathematics\\Iowa State University, \\Ames, IA 50011}
\email{tathagat@iastate.edu}
\urladdr{http://orion.math.iastate.edu/tathagat}

\date{April 15, 2026}

\begin{abstract}
Let $\Gamma$ be a complex reflection group acting on the complex affine or hyperbolic space
$X$ with the set of reflecting hyperplanes $\euH$.
We define an augmented rack $(G, \euK, p)$ associated to the orbifold fundamental group
$G := \pi_1^{\op{orb}}( \Gamma \backslash (X - \euH))$ which 
plays the role of the fundamental rack of a framed link complement
as defined by Fenn and Rourke.
This yields representations of the orbifold fundamental group 
$G$ on the cohomology of the associated rack space.

%The orbifold fundamental group $G$ acts on $\euK$ and on 
%the associated rack space and on the cohomology of
%of rack spaces and this yields representations of the orbifold fundamental group $G$. 
\end{abstract}
\maketitle
%
%***************************************************************************************************************************
%
%
%\medskip

\begin{topic}{\bf Introduction.} Let $\Gamma$ be a complex reflection group acting on 
complex affine or hyperbolic space
$X$ with set of mirrors (i.e. reflecting hyperplanes) $\eu{M}$. 
Assume that there exists an $n \geq 1$ such that for each mirror $a \in \eu{M}$, 
the reflections in $\Gamma$ 
fixing $a$ pointwise forms a cyclic group of order $n$ generated by a complex 
 reflection $R_a$ that acts on the normal bundle to $a$ as multiplication by $e^{2 \pi i/n}$.
Then $\eu{M}$
has the standard rack structure (see example \ref{example-quandle-from-reflection-group})
given by $a  \trr b = R_a b$ for $a, b \in \eu{M}$.
Let $\eu{H}$ be the union of the mirrors in $\eu{M}$.
The purpose of this note is to define an augmented rack $(G, \euK, p)$ 
associated to the orbifold fundamental group
$G := \pi_1^{\op{orb}}( \Gamma \backslash (X - \euH))$. 
%The set of mirrors $\euH$ has the standard rack structure, called Coxeter rack in case $\Gamma$ is a Coxeter group. 
The rack $\euK$ is analog of the fundamental rack of a framed link complement from \cite{FR}.
The rank $\eu{K}$ is to the orbifold fundamental group $G$ what the rack $ \eu{M}$
 is to the reflection group $\Gamma$.
So in the Coxeter type $A_n$, transitioning from $ \eu{M}$ to $\eu{K}$ is 
like transitioning from the permutation 
group $S_n$ to the braid group $B_n$. The orbifold fundamental group
$G$ acts on $\eu{K}$ and on the associated rack space and on the cohomology
of rack space \cite{EG, FRS1, FRS2}
 and this yields representations of $G$. 
\par
For the applications we have in mind, the notion of a fundamental rack from \cite{FR} needs to be modified/generalized because of the
following issues:
\begin{enumerate}
\item We need to generalize to the setting of an orbifold fundamental group.
\item  We need to deal with the fact that the hyperplane arrangement $\euH \subseteq X$
and the corresponding divisor $\Gamma \backslash \euH \subseteq \Gamma\backslash X $
are not smooth real codimension two submanifolds but have singularities.
\item Since the space $X$ and the arrangement 
$\eu{M}$ have more structure in our setup, we
do not need to specify a  ``framing" as a separate piece of data.
\end{enumerate}

Addressing the first issue  is fairly straightforward. 
The main technical tool to address the other two issues is the generalized definitions of a meridian 
$\mu_{c, H}$ around a mirror $H$ based at a point $c$ developed in \cite{AB1, AB2}.
Using this, we find that the definition of \cite{FR} does generalize in the set up described above.
\end{topic}
\begin{definition}
A {\it rack} is a set $\eu{A}$ with a binary operation $(a, b) \mapsto a \trr b$ such that 
\begin{enumerate}
\item For all $a \in \eu{A}$, the map $c \mapsto a \trr c$ is a bijection from $\eu{A}$ to $\eu{A}$.
%In other words for all $a, b \in \eu{A}$ there exists a unique $c \in \eu{A}$ such that $a \trr c = b$.
\item  For all $a, b, c \in \eu{A}$ one has $a \trr (b \trr c) = (a \trr b) \trr (a \trr c)$.  
\end{enumerate}
A {\it quandle} is a rack $(\eu{A}, \trr)$ such that $a \trr a = a$ for all $a \in \eu{A}$.
An {\it augmented rack} is a triple $(G, \eu{A}, p)$ where $G$ is a group, $\eu{A}$ is a $G$-set
 (group actions are always on the left unless otherwise stated
 and are denoted by $(g, a) \mapsto g a$ or $(g, a) \mapsto g \vert_{\eu{A}} a$), and 
 $p: \eu{A} \to G$ is a function such that $p(g a) = g p(a) g^{-1}$ for all $g \in G, a \in \eu{A}$.
 An {\it augmented quandle} is an augmented rack $(G, \eu{A}, p)$ such that 
 $p(a)a = a$ for all $a \in \eu{A}$.
\par
If $(G, \eu{A}, p)$ is an augmented rack (resp. augmented quandle), then 
$\eu{A}$ becomes a rack (resp. a quandle) with the binary operation defined by
$(a \trr b) = p(a) \vert_{\eu{A}} b$.
There is an obvious notion of morphisms of racks (resp. quandles) and augmented racks (resp. augmented quandles).
We remark that there are many equivalent ways of defining a rack and many different notations are used, 
not all of them consistent. Our notation is that of \cite{EG}. To match with notation of \cite{FR} let
$a^b = b \trr a$ (read this as ``$b$ acting on $a$").
%The definition in wikipedia writes $a \tll b$ instead of $a \trr b$.
\end{definition}

Here are a couple of basic examples.
\begin{example} Let $\eu{A}$ be a union of conjugacy classes in a group $G$. Let $G$ act on 
$\eu{A}$ by conjugation
 $g \vert_{\eu{A}} x = g x g^{-1}$.
 Let $p: \eu{A} \to G$ be the inclusion.
 Then $(G, \eu{A}, p)$ is an augmented quandle.
\end{example}
 \begin{example}
 \label{example-quandle-from-reflection-group}
Let $F = \R$ or $\C$. Let $V$ be an inner product space over $F$.
Fix a nontrivial root of unity $\zeta \in F$.
Given a hyperplane $H$ in $V$ the
 (complex) $\zeta$-reflection in $H$, 
 denoted $R_H^{\zeta}$ means an automorphism of $V$
that fixes $H$ pointwise and acts on its orthogonal complement as multiplication by $\zeta$.
 Let $\Gamma \subseteq \op{Aut}(V)$ be a discrete subgroup generated by a set of $\zeta$-reflections.
 If $R$ is a nontrivial reflection in $\Gamma$, then the hyperplane pointwise fixed by
 $R$ is called the {\it mirror} of $R$.
 Let $\eu{M}$ be the set of mirrors of $\zeta$-reflections in $\Gamma$.
Let $p: \eu{M} \to \Gamma$ be the map $p(H) = R_H^{\zeta}$.
 Then $(\Gamma, \eu{M}, p)$ is an augmented quandle. This is a special case of the previous example
 because $\eu{M}$ is in bijection with the set of $\zeta$-reflections
 $\lbrace R^{\zeta}_H \colon H \in \eu{M} \rbrace $ and this set
is an union of conjugacy classes in $\Gamma$ since
 $g R^{\zeta}_H g^{-1} = R^{\zeta}_{g H}$ for all $g \in \Gamma$.
\end{example}

%
%\section{An augmented rack from Hyperplane arrangements}
%
\begin{definition}
\label{def-paths-orthogonal-to-generic-points-of-mirrors}
We define the basic paths we are going to use to define the fundamental rack and 
recall the basic definitions of orbifold fundamental groups and meridians from \cite{AB2}.
Let $X$ be a either the complex affine space $\C^m$  or the complex hyperbolic 
space $\mathbb{B}^{m}_{\C}$ of dimension $m$. 
The discussion below generalizes to more general contexts but the examples we have in mind 
are in $X = \mathbb{B}^m_{\C}$ and the most interesting example lives in $\mathbb{B}^{13}_{\C}$.
Given $a, b $ in $X$, let $\overline{a b}$ 
(resp. $\overline{ab}^{\C}$) denote the real (resp. complex) geodesic segment joining $a$ and $b$ in $X$.
Let $\eu{M}$ be a locally finite collection of totally geodesic 
hypersurfaces in $X$. We'll call the elements of $\eu{M}$ {\it mirrors}.
Let $\eu{H} \subseteq X$ be the union of the mirrors. 
Given $H \in \eu{M}$ and $b \in X - \eu{H}$ there exists a unique point $p$ on $H$ closest to $b$ 
called the projection of $b$ on $H$.
We write $p = \op{pr}_H(b)$.
\par
We'll say a point $x \in \eu{H}$ is a {
\it generic point} if 
$x$ belongs to a single mirror. For $H \in \eu{M}$, let $H^{\circ}$ be the set of generic points of $H$.
Fix a basepoint $\tau \in X - \euH$. For each $t \in [0, 1)$, 
Let $\op{Path}_t(\tau, H^{\circ})$ be the set of paths $\alpha: [0,1] \to X$ such that 
\begin{enumerate}
\item $\alpha(0) = \tau$ and $\alpha[0,1) \subseteq X - \euH$,
\item $p = \alpha(1) \in H^{\circ}$, 
\item there exists $t < 1$ such that 
$p = \op{pr}_H(\alpha(t))$ and $\alpha \vert_{[t, 1]}$ is the real geodesic segment
$\overline{\alpha(t) \alpha(1)} $
\footnote{ in other words $\alpha \vert_{[t, 1]}$ is a real geodesic segment orthogonal to $H$ at $\alpha(1)$}.
It follows that $p = \op{pr}_H( \alpha(s))$ for each $s \in [t, 1]$.
\end{enumerate}
If $t$ satisfies (3), then we'll say $\alpha \vert_{[t, 1]}$ is {\it a final geodesic portion of $\alpha$}.
Note: $\op{Path}_t(\tau, H^{\circ}) \subseteq \op{Path}_{t'}(\tau, H^{\circ})$ if $t \leq t'$.
Let 
\begin{equation*}
\op{Path}(\tau, H^{\circ}) = \cup_{t \in [0, 1)} \op{Path}_t(\tau, H^{\circ}) \text{\; and \;}
\op{Path}(\tau, \eu{H}^{\circ}) = \cup_{H \in \eu{M}} \op{Path}(\tau, H^{\circ}).
\end{equation*}
If $\alpha \in \op{Path}(\tau, H^{\circ})$, 
then we say that $\alpha$ is a {\it path from $\tau$ to a generic point of $H$, orthogonal to $H$}.
Each path $\alpha \in \op{Path}(\tau, \eu{H}^{\circ})$ uniquely determines the mirror at
which it ends, so we have a well defined map
\begin{equation*}
e: \op{Path}(\tau, \eu{H}^{\circ}) \to \eu{M} \text{\; given by\;}  e(\alpha) = H \text{\; if \;} \alpha(1) \in H.
\end{equation*}
\par
Say that $\alpha, \beta \in \op{Path}(\tau, H^{\circ})$ are homotopic, written $\alpha \sim \beta$, if there exists a map $F:  [0,1]^2 \to X$
such that $F(0, \cdot) = \alpha$, $F(1,\cdot) = \beta$, and there exists $s \in [0, 1)$ such that
$F(t, \cdot) \in \op{Path}_s(\tau,H^{\circ})$ for all $t \in [0,1]$.
So $\sim$ defines an equivalence relation on each $\op{Path}(\tau, H^{\circ})$
and hence on their disjoint union $\op{Path}(\tau, \eu{H}^{\circ})$.
Write
\begin{equation*}
\Pi(\tau, H^{\circ}) = \op{Path}(\tau, H^{\circ})/\sim \text{\; and \;}
\eu{K} = \Pi( \tau, \euH^{\circ}) = \op{Path}(\tau, \eu{H}^{\circ})/\sim.
\end{equation*}
So an element of $\Pi(\tau, H^{\circ})$ is a homotopy class of paths from $\tau$ to generic points of $H$, 
orthogonal to $H$ and $\eu{K}  = \Pi( \tau, \euH^{\circ})$ is the disjoint union of these homotopy classes.
Each homotopy class determines the hyperplane at which it ends, so we have a well defined map
\begin{equation*}
e : \eu{K} \to \eu{M} \text{\; such that \;} e([\alpha]) = H \text{\; if \;} \alpha(1) \in H.
\end{equation*}
\par
Assume that for each $H \in \eu{M}$ we have a unique finite order automorphism $R_H \in\op{Aut}(X)$  
(the automorphism group of $X$) such that $R_H(\eu{M}) = \eu{M}$ and $R_H$
pointwise fixes
$H$, and acts as an anti-clockwise rotation of $2 \pi/n$ on the normal bundle to $H$,
where $n$ is the order of $R_H$. By a reflection with mirror $H$, we mean a non-trivial element of $\op{Aut}(X)$ 
of the form  $R_H^j$. 
Note that the uniqueness in the definition of $R_H$ implies  that $R_{s H} = s R_H s^{-1}$ for all $s \in \op{Aut}(X)$ 
and $H \in \euH$.
Assume that the reflections 
$\lbrace R_H \colon H \in \euH \rbrace$ generate a discrete subgroup of $\op{Aut}(X)$, denoted by $\Gamma$.
We call $\Gamma$ the reflection group of $(X, \eu{H})$. 
Let 
\begin{equation*}
G = \pi_1^{\op{orb}}(\Gamma \backslash (X - \eu{H}) , \tau)
\end{equation*}
be the orbifold fundamental group of $\Gamma \backslash (X - \eu{H})$ based at $\tau$.
The elements of $G$ are pairs $(\gamma, s)$
where $s \in \Gamma$ and $\gamma$ is a homotopy class of paths from $\tau$ to $s \tau$ in $X - \euH$
with multiplication defined by
\begin{equation*}
(\gamma, s) (\gamma_1, s_1) = (\gamma * s \gamma_1, s s_1).
\end{equation*}
Here $\gamma * s \gamma_1$ means $\gamma$ followed by $s \gamma_1$.
Note that $(\gamma, s)^{-1} = (s^{-1} \gamma^{op}, s^{-1})$ where $\gamma^{op}$ denotes the
opposite path of $\gamma$.
Observe that the group $G$ acts on $\eu{K}$ by 
\begin{equation*}
(\gamma, s) \alpha = \gamma * s(\alpha) \text{\; for } (\gamma, s) \in G, \text{\; and \;} \alpha \in \eu{K}.
\end{equation*}
\par
Finally we recall from \cite{AB1, AB2} the definition of a {\it meridian}
$\mu_{b, H}$ based at a point $b \in X - \eu{H}$ and going around a mirror $H$
(or rather a special case of the definition that is sufficient for our purpose).
Let $H$ be a mirror and let $b \in X - \eu{H}$ such that $p = \op{pr}_H(b)$ is a generic
point of $H$ and such that the real geodesic $\overline{b p}$ does not meet any mirror other than $H$.
Then  the meridian $\mu_{b, H}$ is (the homotopy class of) the path defined as follows:
\par
 Choose a point $d$ on the geodesic $\overline{b p}$ 
sufficiently close to $p$ such that $d$ is contained in in a ball centered at $p$ that meets
no mirror other than $H$. 
Let $\op{arc}(d, R_H(d))$
 be the semicircular arc centered at $p$  lying in  the complex geodesic $\overline{b p}^{\C}$ 
going from $d$ to $R_H(d)$. This arc makes a counterclockwise rotation of angle $2 \pi /n$ around $p$.
The meridian $\mu_{b, H}$ is concatenation of three paths:
\begin{equation*}
\mu_{b, H} = \overline{b d} *\op{arc}(d, R_H(d)) * R_H( \overline{d b}).
\end{equation*}
We'll call the second part the {\it semicircular part of $\mu_{b, H}$}.
%The name meridian is borrowed from knot theory.
Lemma 3.1 of \cite{AB2} shows that the homotopy class of the path $\mu_{b,H}$
does not depend on the choice of $d$, so  $(\mu_{b, H}, R_H)$ yields a well defined loop in 
 $ \pi_1^{\op{orb}}(\Gamma \backslash (X - \eu{H}), b)$ 
that ``goes around the hypersurface represented by the mirror $H$ once counterclockwise". 
The notation $\mu_{b, H}$ may sometimes denote this homotopy class (also called a meridian)
or any path in this homotopy class.
\end{definition}
\begin{definition}
\label{def-the-path-mu-alpha}
Now we come to the main definition. Given a mirror $H \in \euH$ and a path $\alpha \in \op{Path}(\tau, H^{\circ})$,
we define a (homotopy class of) path $\mu_{\alpha}$ in $X - \eu{H}$ starting at $\tau$ and ending at $R_H(\tau)$.
We write $p = \alpha(1)$. Recall that if $t$ is close to $1$, then $p$ is the projection of $\alpha(t)$ on $H$.
To define this path, choose a point $c = \alpha(t)$ on the path $\alpha$ that is close enough to $H$
such that:
\begin{enumerate}
\item $c$ is in a final geodesic portion of $\alpha$, i.e.
$\alpha \vert_{[t, 1]} = \overline{c p}$,
\item  $c$ (and hence $\alpha \vert_{[t, 1]}$) is contained in a ball $B$ centered at $p$ that meets 
no mirror other than $H$.
\end{enumerate}
%Note that the first condition can be satisfied because of the third condition satisfied by $\alpha$
%given in definition \ref{def-paths-orthogonal-to-generic-points-of-mirrors}.
Note that the second condition can be satisfied because $p$ is a generic point of $H$ and the collection
$\eu{H}$ is locally finite.
Let $\alpha_c$ denote the restriction of $\alpha$ to $[0, t]$. Define $\mu_{\alpha}$ as a concatenation of three paths:
\begin{equation}
\mu_{\alpha} = \alpha_{c} * \mu_{c,H} * R_H( \alpha_c^{op}).
\label{eq-definition-of-mu-alpha}
\end{equation}
Here $\mu_{c, H}$ represents any path in the class of the meridian $\mu_{c, H}$.
The homotopy class of  $\mu_{\alpha}$ is well defined and does not depend on the choice of the point $c$.
All the technical work for this  is done in \cite{AB1, AB2}
where this and much more is carefully proved in the context of meridians. We collect what we need in the lemma
below. We remark that the notion of meridian in \cite{AB1, AB2} is more general 
because of many possible complications\footnote{For instance in \cite{AB1, AB2} the base-point is allowed to be a whole contractible subset 
(called ``fat base-point"), the point $p$ towards which the meridian is ``originally aiming" 
is allowed to be a non-generic point of a mirror,
and finally the various paths involved may meet mirrors along the way in which case 
one has to carefully define detours around these mirrors.
%and then one has to check well-definedness up-to homotopy in this more general setting.
}
 but thankfully none of these complications arise in our set up. In particular in the definition of $\mu_{c, H}$  above,
 no ``semicircular detours around mirrors" is
 necessary because $\overline{c p}$ just follows the path $\alpha$, so $\overline{c p}$ does not
 meet any mirror before hitting $H$.
\end{definition}

%The definition of $\mu_{c, H}$ and it basic properties developed in \cite{AB1, AB2}is the main technical ingredient we need. 
\begin{lemma}
\label{l-homotopy-invariance}
Assume the setup in the definition \ref{def-the-path-mu-alpha}.
\begin{enumerate}
\item The homotopy class of the path $\mu_{\alpha}$ in $(X - \eu{H})$ rel endpoints
do not depend on the choice of the point $c$. So $(\mu_{\alpha} , R_H)$ is an well defined element of the orbifold fundamental
group $\pi_1(\Gamma \backslash(X - \eu{H}), \tau)$.
\item If $\alpha$ and $\beta$ in $\op{Path}(\tau,H^{\circ})$
are homotopic, i.e. determine the same element of $\eu{K}$, then $\mu_{\alpha}$ and $\mu_{\beta}$ are homotopic in 
$X - \eu{H}$.
\end{enumerate}
\end{lemma}
\begin{proof}
(1) Let $c, c'$ be two choices of $c$. Without loss assume $c'$ is closer to $p$ than $c$. Then the two paths in questions are
$\alpha_{c} * \mu_{c,H} * R_H( \alpha_c^{op})$
and 
$\alpha_{c'} * \mu_{c', H} * R_H( \alpha_{c'}^{op})$.
Since both $c$ and $c'$ are in a final geodesic portion of $\alpha$,
the only possible difference between the two paths is that 
 the semicircular portion of the meridians may be taken starting at different points $d, d'$.
So part (1) follows from well-definedness of the homotopy class of the meridian (see Lemma 3.1 of \cite{AB2}).
\par
(2) This is similar to the proofs of lemmas 3.1 and 3.7 in \cite{AB2}.
For completeness sake, we indicate an argument.
Figure \ref{fig1} shows the various paths and points involved in the argument
(compare with figures 3.2 or 3.4 in \cite{AB2}).
Let $\lbrace \alpha_t \colon t \in [0, 1] \rbrace$ be a homotopy from
$\alpha = \alpha_0$ to $\beta = \alpha_1$, in
particular, each $\alpha_t \in \op{Path}_s(\tau, H)$ for some $s < 1$. 
Write $c_t = \alpha_t(s)$ and $p_t = \alpha_t(1)$.
Then $\op{pr}_H( c_t ) = p_t$  and $\alpha_t \vert_{[s, 1]}$ is the
geodesic segment $\overline{c_t p_t}$ 
for all $t \in [0, 1]$.  
Note that $\delta_0 = \min \lbrace d( c_t, p_t) \colon t \in [0, 1] \rbrace > 0$ and
 that $E = \lbrace p_t \colon t \in [0, 1] \rbrace$ is a compact subset of $H^{\circ}$.
So we can choose $\epsilon > 0$ such that
$\epsilon < \delta_0$ and such that 
no mirror other than $H$ comes within distance $\epsilon$ of $E$.
So each $\alpha_t$ has a final geodesic portion of length at least $\epsilon$.
Now consider the
tube $T$ around $E$ obtained by taking the union of the radius $\epsilon$ discs in the complex geodesics 
$\overline{c_t p_t}^{\C}$ for all $t \in [0, 1]$.
Note that each $c_t$ lies outside this tube by choice of $\epsilon$.
Let $d_t$ be the point on the boundary $\partial T$ where the geodesic $\overline{c_t p_t}$ enters $T$.
Then $\mu_{\alpha_t}$ has a representative $\mu_{\alpha_t}^*$ that follows $(\alpha_t)_{d_t}$,  then
takes its semicircular detour along $\partial T$ going from $d_t$ to $d'_t = R_H(d_t)$ 
lying in the complex geodesic $\overline{c_t p_t}^{\C}$
and then follows
$R_H( (\alpha_t)_{d_t}^{op})$. As $t$ varies from $0$ to $1$, these $\mu_{\alpha_t}^*$'s define a
homotopy from $\mu_{\alpha_0}^*$ to $\mu_{\alpha_1}^*$.
By part (1), the class of $\mu_{\alpha_0}^*$ is $\mu_{\alpha}$ and
the class of $\mu_{\alpha_1}^*$ is $\mu_{\beta}$.
 \end{proof}
 \begin{figure}
 \centerline{
\begin{tikzpicture}[scale = 1.2]
\foreach \t in {0, .05,...,1.25}{ 
\draw[line width = 1.8pt, gray!40!white, opacity = .3]  plot[domain=0 : 360, samples=20, variable = \u] ( {0.5 * cos(\u)}, {\fnf(1, \t) + .5 * \t }, {0.5 * sin(\u)} );  
}
\foreach \t in {0, .1,...,1.2}{ 
\draw[line width = .5pt, gray!80!white, opacity = .6]  plot[domain=-2 : 1, samples=20, variable = \x] (0,  {\fnf(\x, \t)}, 2.5 - \x);  
\draw[line width = .5pt, gray!80!white, opacity = .6]  plot[domain=1 : {1.5 - sin(180*\t)/2.7}, samples=20, variable = \x] (0, {\fnf(1, \t) + .5 * \t *(\x - 1)}, 2.5 - \x ); 
\draw[line width = .5pt, gray!80!white, opacity = .6]  plot[domain={1.5 - sin(180*\t)/2.7}: 2, samples=20, variable = \x] (0, {\fnf(1, \t) + .5 * \t *(\x - 1)}, 2.5 - \x ); 
\draw[line width = .5pt, gray!80!white, dotted, opacity = .6]  plot[domain=2 : 2.5, samples=20, variable = \x] (0, {\fnf(1, \t) + .5 * \t *(\x - 1)}, 2.5 - \x );  
\draw [fill = gray!80!white] (0, {\fnf(1, \t) + .5 * \t *( (\fns(\t)) - 1)}, {2.5 -  (\fns(\t)) } )  circle [radius=.01cm];
\draw[line width = .5pt, gray!80!white, opacity = .6]  plot[domain=0 : 90, samples=20, variable = \u] ( {0.5 * cos(\u)}, {\fnf(1, \t) + .5 * \t }, {0.5 * sin(\u)} );  
\draw[line width = .5pt, gray!80!white, opacity = .6]  plot[domain=-2 : 1, samples=20, variable = \x] (2.5 - \x,  {\fnf(\x, \t)}, 0);  
\draw[line width = .5pt, gray!80!white, opacity = .6]  plot[domain=1 : {\fns(\t)}, samples=20, variable = \x] (2.5 - \x, {\fnf(1, \t) + .5 * \t *(\x - 1)}, 0 );  
\draw[line width = .5pt, gray!80!white, opacity = .6]  plot[domain={\fns(\t)} : 2, samples=20, variable = \x] (2.5 - \x, {\fnf(1, \t) + .5 * \t *(\x - 1)}, 0 );  
\draw[line width = .5pt, gray!80!white, dotted, opacity = .6]  plot[domain=2 : 2.5, samples=20, variable = \x] (2.5 - \x, {\fnf(1, \t) + .5 * \t *(\x - 1)}, 0 );  
\draw [fill = gray!80!white] ({2.5 -  (\fns(\t)) }, {\fnf(1, \t) + .5 * \t *( (\fns(\t)) - 1)}, 0 )  circle [radius=.01cm];
}
\foreach \t in {0.6}{ 
\draw[line width = .6pt, gray!70!black, opacity = 1]  plot[domain=-2 : 1, samples=20, variable = \x] (0,  {\fnf(\x, \t)}, 2.5 - \x);  
\draw[line width = .6pt, gray!70!black, opacity = 1]  plot[domain=1 : {1.5 - sin(180*\t)/2.7}, samples=20, variable = \x] (0, {\fnf(1, \t) + .5 * \t *(\x - 1)}, 2.5 - \x );  
\draw[line width = .6pt, gray!70!black, opacity = 1]  plot[domain={1.5 - sin(180*\t)/2.5}: 2, samples=20, variable = \x] (0, {\fnf(1, \t) + .5 * \t *(\x - 1)}, 2.5 - \x );  
\draw[line width = .6pt, gray!70!black, dotted, opacity = 1]  plot[domain=2 : 2.5, samples=20, variable = \x] (0, {\fnf(1, \t) + .5 * \t *(\x - 1)}, 2.5 - \x );  
\draw [fill = gray!80!black] (0, {\fnf(1, \t) + .5 * \t *( (\fns(\t)) - 1)}, {2.5 -  (\fns(\t)) } )  circle [radius=.01cm];
\draw[line width = .6pt, gray!70!black, opacity = 1]  plot[domain=0 : 90, samples=20, variable = \u] ( {0.5 * cos(\u)}, {\fnf(1, \t) + .5 * \t }, {0.5 * sin(\u)} );  
\draw[line width = .6pt, gray!70!black, opacity = 1]  plot[domain=-2 : 1, samples=20, variable = \x] (2.5 - \x,  {\fnf(\x, \t)}, 0);  
\draw[line width = .6pt, gray!70!black, opacity = 1]  plot[domain=1 : {\fns(\t)}, samples=20, variable = \x] (2.5 - \x, {\fnf(1, \t) + .5 * \t *(\x - 1)}, 0 );  
\draw[line width = .6pt, gray!70!black, opacity = 1]  plot[domain={\fns(\t)} : 2, samples=20, variable = \x] (2.5 - \x, {\fnf(1, \t) + .5 * \t *(\x - 1)}, 0 );  
\draw[line width = .6pt, gray!70!black, dotted, opacity = 1]  plot[domain=2 : 2.5, samples=20, variable = \x] (2.5 - \x, {\fnf(1, \t) + .5 * \t *(\x - 1)}, 0 ); 
\draw [fill = gray!80!black] ({2.5 -  (\fns(\t)) }, {\fnf(1, \t) + .5 * \t *( (\fns(\t)) - 1)}, 0 )  circle [radius=.01cm]; 
}
\draw[line width = 1pt, gray!80!white, opacity = .8] (0, 0, 0) -- (0, 3, 0); % the central vertical line representing the mirror H
%  
% below we draw various points and label them
%
 \draw [fill = gray] (0,0,4.5) circle [radius=.02cm]; \node [left] at (0, 0, 4.5) {\tiny $\tau$};
 \draw [fill = gray] (4.5,0,0) circle [radius=.02cm]; \node [right] at (4.5, 0, 0) {\tiny $\tau'$};
 \draw [fill = gray] ( 0, {\fnf(1, .6) + .5 * .6 }, .5 ) circle [radius=.02cm]; \node [left] at (.03, {\fnf(1, .6) + .5 * .6 }, .36 ) {\tiny $d_t$};
 \draw [fill = gray] ( 0.5, {\fnf(1, .6) + .5 * .6 }, 0 ) circle [radius=.02cm]; \node [right] at ( 0.43, {\fnf(1, .6) + .5 * .6 + .1}, 0 ) {\tiny $d'_t$};
 \draw [fill = gray] (0, {\fnf(1, .6) + .5 * .6 *( (\fns(.6)) - 1)}, {2.5 -  (\fns(.6)) } ) circle [radius=.02cm]; 
 \node [left] at  (0.07, {\fnf(1, .6) + .5 * .6 *( (\fns(.6)) - 1) + .05}, {2.5 -  (\fns(.6)) } )  {\tiny $c_t$};
 \draw [fill = gray] ({2.5 -  (\fns(.6)) }, {\fnf(1, .6) + .5 * .6 *( (\fns(.6)) - 1)}, 0 ) circle [radius=.02cm]; 
 \node [right] at  ({2.5 -  (\fns(.6)) -.1}, {\fnf(1, .6) + .5 * .6 *( (\fns(.6)) - 1) + .1}, 0 )  {\tiny $c'_t$};
 \draw [fill = gray]  (0, {\fnf(1, .6) + .5 * .6 *(2.5 - 1)},  0) circle [radius=.02cm]; \node [right] at (-.07, {\fnf(1, .6) + .5 * .6 *(2.5 - 1) + .05 + .02},  0) {\tiny $p_t$};
\node [right] at (-.05, .1, 0) {\tiny $H$};
\end{tikzpicture}
}
\caption{The homotopy appearing in proof of Lemma \ref{l-homotopy-invariance}(b).
The paths at a particular time $t$ are emphasized.
The central vertical line is the mirror $H$.
 We write $\tau' = R_H(\tau)$, $d_t' = R_H(d_t)$ etc.
 The lowest path is $\mu_{\alpha}$ and the highest path is $\mu_{\beta}$
 The path going from $\tau$ to $p_t$ is $\alpha_t$ (with the last portion drawn as a dotted line). 
 The path going from $\tau$ to $d_t$ to $d'_t$ to $\tau'$ is 
 $\mu_{\alpha_t}^*$. Its middle portion from $c_t$ to $c'_t$ is identical to a meridian $\mu_{c_t, H}$.
The central circular arc from $d_t$ to $d'_t$ lies on the
boundary of the shaded tube $T$.
}
\label{fig1}
\end{figure}
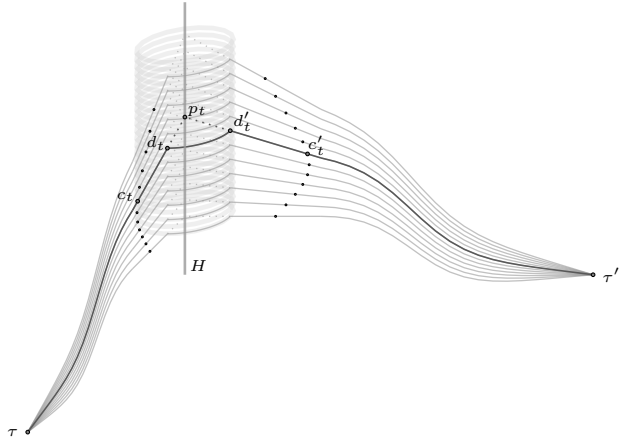
 \begin{definition}
Definition \ref{def-the-path-mu-alpha} and lemma \ref{l-homotopy-invariance}
yield a well defined map $\alpha \mapsto \mu_{\alpha}$ 
from $\Pi(\tau, H^{\circ})$ to homotopy class of paths in $X - \euH$
starting at $\tau$ and ending at $R_H(\tau)$.
Define $p: \eu{K} \to G$ by 
\begin{equation*}
p(\alpha) = (\mu_{\alpha}, R_{H} ) \text{\; where \;} H = e(\alpha).
\end{equation*}
\end{definition}
\begin{theorem}
\label{th-rack-structure}
Assume $G$ is generated by the meridians based at $\tau$.
The triple $(G, \eu{K}, p)$ is an augmented rack, in fact an augmented quandle. In particular 
$\eu{K}$ is a rack under the operation 
$\alpha \triangleright \beta = p(\alpha) \beta$.
\end{theorem}
\begin{proof}
Take $g = (\gamma, s) \in G$ where $s \in \Gamma$ and $\gamma$ is a path in 
 $(X - \eu{H})$ from $\tau$ to $s \tau$
 and take a path $\alpha$ representing an element of $\eu{K}$.
We need to check $p( g \alpha) = g p(\alpha) g^{-1}$. Let $H = e(\alpha)$.
The paths involved in the argument are shown in figure \ref{fig2}.
%The main path  $\mu_{g \alpha}$  is the path going from $\tau$ to $R_{H'}(\tau)$.
%
 \begin{figure}
 \centerline{
\begin{tikzpicture}[scale = 1.2]
%
%Draw the path $\gamma$
%
\draw[line width = .5pt, ->] (-1, 0) -- (-.3, -.1) -- (.3, -.1); 
\draw[line width = .5pt]  (.3, -.1) -- (1, 0); 
\node [above] at (.3, 0) {\tiny $\gamma$};
%\draw[line width = 1.2pt, gray!80!white, opacity = .8] (0, -1) -- (0, -.15) (0, 0) -- (0, 1);   
%\node [below] at (0, -1) {\tiny $s^{\bot}$};
%  
% Draw the path $\mu_{\alpha}$
%
\draw[line width = .5pt,  ->]  (-1,0)  to[out=90, in=0]  (-1.8, 0); 
\draw[line width = .5pt, ] (-1.8, 0) -- (-1.9, 0) arc[start angle= 30, end angle=150, radius=.123cm] -- (-2.1, 0) --(-2.2, 0)  to[out=180, in= 270] (-3, 0);
\node [above] at (-1.6, 0) {\tiny $\mu_{\alpha}$};
\draw[line width = 1.2pt, gray!80!white, opacity = .8] (-2, -1)  -- (-2, 1); 
\node [left] at (-2, -1) {\tiny $H$};
%
%Draw the path $s \mu_{\alpha}$ mostly by rotating $mu_{\alpha}$.
%
\begin{scope}[xscale = -1, yscale = -1]
\draw[line width = .5pt, ->]  (-1,0)  to[out=90, in=0]  (-1.8, 0); 
\draw[line width = .5pt] (-1.8, 0) -- (-1.9, 0) arc[start angle= 30, end angle=150, radius=.123cm] -- (-2.1, 0) --(-2.2, 0)  to[out=180, in= 270] (-3, 0);
\node [above] at (-1.6, 0) {\tiny $s \mu_{\alpha}$};
\draw[line width = 1.2pt, gray!80!white, opacity = .8] (-2, -1)  -- (-2, 0) (-2, .15) -- (-2, 1); 
\node [left] at (-2, 1) {\tiny $H'$};
\end{scope}
%
%Draw the path $R_{H'}(\gamma)$ mostly by shifting $\gamma$
%
\begin{scope}[xshift = 4cm]
\draw[line width = .5pt, ->] (-1, 0) -- (-.3, .1) -- (.3, .1); 
\draw[line width = .5pt]  (.3, .1) -- (1, 0); 
\node [above] at (.4, .15) {\tiny $R_{H'}(\gamma^{op})$};
%\draw[line width = 1.2pt, gray!80!white, opacity = .8] (0, -1) -- (0, 1); 
%\node [below] at (0, -1) {\tiny $R_{H'}(s^{\bot})$};
\end{scope}
%
% below we draw various points and label them
%
 \draw [fill = gray] (-1, 0) circle [radius=.02cm]; \node [below] at (-1,0) {\tiny $\tau$};
\draw [fill = gray] ( 1, 0) circle [radius=.02cm]; \node [below] at (.9,0) {\tiny $\tau'$};
 \draw [fill = gray] (-3, 0) circle [radius=.02cm]; \node [below] at (-3.2, -.1) {\tiny $R_H(\tau)$};
 \draw [fill = gray] (3, 0) circle [radius=.02cm]; \node [below] at (3, 0) {\tiny $R_{H'}(\tau') $};
 \draw [fill = gray] (5, 0) circle [radius=.02cm]; \node [below] at (5, 0) {\tiny $R_{H'}(\tau)$}; 
\end{tikzpicture}
}
\caption{A schematic diagram of the paths involved in the proof of theorem \ref{th-rack-structure}.
For simplify the drawing, we picture $s$ and $R_H$ as complex reflections of order $2$.
The path $\gamma$ is indicated by a piecewise straight line 
from $\tau$ to $\tau' = s \tau$.
The path $\mu_{\alpha}$ is indicated by the curved path going from $\tau$ to $R_{H}(\tau)$. 
}
\label{fig2}
\end{figure}
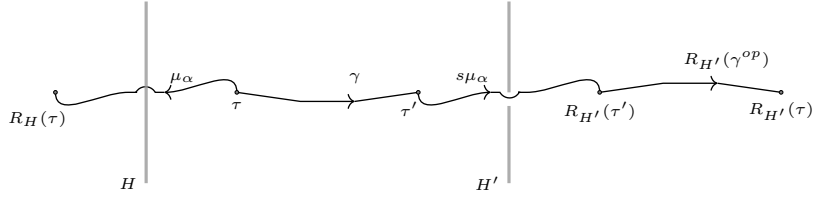
Note that
\begin{equation*}
\mu_{\alpha} = \alpha_c * \mu_{c, H} * R_H(\alpha_c^{op}).
\end{equation*}
Let us write $\alpha' = s \alpha$ and $c' = s c$ and $H' = s H$. So $R_{H'} = s R _H s^{-1}$.
We have
\begin{equation*}
s \mu_{\alpha} 
=  \alpha'_{c'} *  \mu_{c',H'} * s R_H  \alpha_c^{op}  
 =  \alpha'_{c'} *  \mu_{c',H'} * R_{H'} (\alpha'_{c'})^{op}
%&= (s \alpha)_{c} * \mu_{c, M} *  R_M (s \alpha)_c^{op}.
\end{equation*}
where the second equality holds since
\begin{equation*}
s R_H \alpha_c^{op} = s R_H s^{-1} s \alpha_c^{op} = R_{H'} (\alpha'_{c'})^{op}.
\end{equation*}
Note that $g \alpha = (\gamma, s) \alpha = \gamma * s \alpha = \gamma * \alpha'$ and 
$e( g \alpha) = H'$. To write a path representing $\mu_{g \alpha}$ 
 (in figure \ref{fig2} this is the path going from $\tau$ to $R_{H'}(\tau)$)
we need to choose
a point on $\alpha'$ close to $H'$. We can choose this point to be $c'$. 
Then 
\begin{align*}
\mu_{g \alpha} 
&= (\gamma * \alpha')_{c'} *\mu_{c', H'} * R_{H'} (\gamma * \alpha')^{op}_{c'}  \\
&= \gamma * \alpha'_{c'} * \mu_{c', H'} * R_{H'} (  \alpha'_{c'})^{op} * R_{H'} \gamma^{op} \\
&= \gamma * s \mu_{\alpha} * R_{H'} \gamma^{op}.
\end{align*}
Finally we compute:
\begin{align*}
g p(\alpha)g^{-1} 
%&= (\gamma, s) (\mu_{\alpha}, R_H) (s^{-1} \gamma^{op}, s^{-1}) \\
%&= (\gamma, s) (\mu_{\alpha} * R_H s^{-1} \gamma^{op}, R_H s^{-1} ) \\
&=  (\gamma * s \mu_{\alpha} * s R_H s^{-1} \gamma^{op}, s R_H s^{-1} ) \\
&=  (\gamma * s \mu_{\alpha} * R_{H'} \gamma^{op}, R_{H'} ) \\
&=  (\mu_{g \alpha}, R_{H'} ) \\
& = p(g \alpha).
\qedhere
\end{align*}
It is easy to verify that $p(\alpha) \alpha= \alpha$.
\end{proof}
\begin{remark}
We may relax the condition (3) in \ref{def-paths-orthogonal-to-generic-points-of-mirrors} and just say that
the path $\alpha$ is orthogonal to the mirror $H$ at $\alpha(1)$. Since $\alpha(1)$ is a generic point of $H$, 
this defines exactly the same homotopy classes of paths, so the later definitions do not change.
Choosing $\alpha$ to follow the geodesic near $ t = 1$ is a convenience because it means that 
if we choose a point $c$ close to $\alpha(1)$, then from $c$ onwards $\alpha$ and the meridian
$\mu_{c, H}$ follows the same path until $\mu_{c, H}$ starts to go around $H$; this makes the
definition of the path $\mu_{\alpha}$ in \ref{def-the-path-mu-alpha} somewhat simpler.
\end{remark}
\begin{example}
We end by discussing three interesting examples in which our construction applies. The first example is our main reason for
developing the definitions here. This example is related to the monstrous
proposal conjecture  \cite{A}.
\par
(1) Let $M$ be the monster simple group.
The monstrous proposal conjecture states that the group $(M \times M) \rtimes \Z/2$,
is an explicit quotient of an orbifold fundamental group of a complex thirteen dimensional ball quotient.
Below we describe the main characters involved in this conjecture. 
For more details,  we refer the reader to the short survey
 \cite{A} or the article \cite{AB3}.
  An Eisenstein lattice means a free $\Z[e^{2 \pi i/3}]$-module
with a $\Z[e^{2 \pi i/3}]$-valued nonsingular hermitian form.
We start with a special Eisenstein lattice $L$ of signature $(13, 1)$; 
one explicit description of $L$ is ``the complex Leech lattice plus a hyperbolic cell". 
The automorphism group of $L$, denoted by $\Gamma$, is a finite covolume discrete subgroup
of $U(13, 1)$.
The group $\Gamma$ contains order $3$ complex reflections in all the (infinitely many) shortest positive norm
vectors of $L$ and is generated by these complex reflections.
The complex hyperbolic reflection group $P\Gamma \subseteq PU(13, 1)$  naturally acts 
on the complex thirteen ball $X = B^{13}_{\C}$.
% preserving is unique $U(13, 1)$ invariant
%negative curvature metric (also known as the complex hyperbolic space).
Let $\eu{H}$ be the union of the mirrors of $\Gamma$ acting on $X$.
\par
Think of the incidence graph $D$ of the finite projective plane $P^2 \mathbb{F}_3$
as a Coxeter diagram.
Ivanov-Norton-Conway etal \cite{CNS, I, N, CS} proved that  
$(M \times M) \rtimes \Z/2$ is the quotient of the Coxeter group of $D$ 
by some natural looking relations (called deflation relations).
%The same graph $D$ appears in a complex hyperbolic reflection group as follows.
Now it turns out that
there exists a point $\tau \in X$ whose stabilizer  in $P\Gamma$ is 
the automorphism group of the graph $D$, such that $\tau$ has
 exactly $26$ mirrors $H_1, \dotsb, H_{26}$ closest to it, and such that
order $3$ complex reflections $R_1, \dotsb, R_{26}$ in these mirrors generate $\Gamma$ (see \cite{B}).
These $26$ complex reflections can be indexed by the vertices of $D$ 
such that they satisfy the braiding and commuting relations of $D$.
In \cite{AB1, AB2, AB3} it is proved that the orbifold fundamental group 
$G = \pi_1^{\op{orb}}( \Gamma \backslash (X - \eu{H}), \tau)$ is generated by
the $26$ meridians $g_1 = (\mu_{\tau, H_1}, R_1), \dotsb, g_{26} = (\mu_{\tau, H_{26}}, R_{26})$
and these meridians satisfy the braiding and commuting relations of $D$ and 
also satisfy the deflation relation. The monstrous proposal conjecture states that
the quotient $Q$ of  the orbifold fundamental group $G$ by the relations $g_1^2 = \dotsb = g_{26}^2 = 1$ 
is $(M \times M) \rtimes \Z/2$. In \cite{AB3} it is proved that
this quotient $Q$ is either $(M \times M) \rtimes \Z/2$ or $\Z/2$. Ruling out the second case would
prove the monstrous proposal conjecture.
\par
If true, the monstrous proposal conjecture leads to construction of a complex $12$-manifold 
on which the monster acts by deck transformation (see \cite{A}). 
Further evidence for existence of such a manifold is provided by \cite{L}.
This manifold seems to be a great potential candidate for the ``monster manifold"
sought by Hirzebruch \cite{HBJ}. 
These connections motivated us to define the rack $\eu{K}$ to build representations of $G$ 
(for example on the rack cohomology of $\eu{K}$).
%Computations in elliptic genera suggest that
%the action of the monster on the appropriate cohomology spaces of
%the ``correct monster manifold" should give many of the moonshine representations
%of the monster.
\par
(2) There is a subgroup $\Gamma_{DM} \subseteq \Gamma$
acting on a complex $9$-ball $X_{DM}$ with its own union of mirrors $\eu{H}_{DM}$ such that
$(\Gamma_{DM} \backslash( X_{DM} - \eu{H}_{DM} )$ is the space of $12$-distinct points on $\mathbb{P}^1$.
The complex reflection group $\Gamma_{DM}$ is the largest example found by Deligne-Mostow 
among the examples of lattices in $PU(n,1)$ constructed via monodromy of Lauricella hypergeometric functions
(see \cite{DM, M}).
In fact the lattice $\Gamma_{DM}$ and the moduli interpretation of its ball quotient is
the key ingredient in the proof of the deflation relation in \cite{AB3}.
\par
(3) The moduli space of smooth cubic surfaces in $\mathbb{P}^3$ can be realized as a ball quotient
$\Gamma_4 \backslash (\mathbb{B}^4_{\C} - \eu{H}_4)$ of the sort described above (see \cite{ACT}).
Here $\Gamma_4$ is the automorphism group of the unique
self-dual Eisenstein lattice of signature $(4, 1)$ and contains order $6$ complex reflections in its roots.
Much of the complex hyperbolic geometry associated to this example is similar to the ones described
in the first example. This seems to lead to a nice description of the 
fundamental group of the space of smooth cubic surfaces. 
%This is work in progress jointly with  Allcock and Looijenga.
\end{example}

%
%***************************************************************************************************************************
%

\end{document}